\documentclass[a4paper]{amsart}

\usepackage[T1]{fontenc}
\usepackage[english]{babel}
\usepackage[alphabetic]{amsrefs}
\usepackage{amssymb}
\usepackage{ytableau}
\usepackage{geometry}
\usepackage{fullpage}
\usepackage[shortlabels]{enumitem}
\usepackage{graphicx}
\usepackage{hyperref}
\usepackage{mathrsfs}
\usepackage{mathtools}

\newcommand{\B}{\mathbb{B}}
\newcommand{\C}{\mathbb{C}}
\renewcommand{\H}{\mathbb{H}}
\newcommand{\Z}{\mathbb{Z}}
\newcommand{\DP}{\operatorname{DP}}
\newcommand{\LD}{\operatorname{LD}}
\newcommand{\dinv}{\operatorname{dinv}}
\newcommand{\area}{\operatorname{area}}
\newcommand{\comp}{\operatorname{comp}}
\newcommand{\doff}{\operatorname{doff}}

\newtheorem{theorem}{Theorem}[section]
\newtheorem{lemma}[theorem]{Lemma}
\newtheorem{proposition}[theorem]{Proposition}
\newtheorem{corollary}[theorem]{Corollary}
\newtheorem{conjecture}[theorem]{Conjecture}
\newtheorem*{maintheorem}{Main Theorem}
\newtheorem{example}[theorem]{Example}

\newtheorem{definition}[theorem]{Definition}
\numberwithin{equation}{section}

\renewcommand{\tbinom}[2]{\genfrac{[}{]}{0pt}{}{#1}{#2}_t}
\DeclareMathOperator{\Res}{Res}

\allowdisplaybreaks[4]

\title{Schur positivity of the nabla operator on two-column modified Hall--Littlewood polynomials}

\author{Menghao Qu}
\address{Scuola Normale Superiore\\
	Piazza dei Cavalieri 7,
	56126 Pisa\\ Italy}\email{menghao.qu@sns.it}

\begin{document}
\begin{abstract}
Bergeron, Garsia, Haiman, and Tesler (1999) introduced two Schur positivity conjectures concerning the action of the $\nabla$ operator on modified Hall--Littlewood polynomials. Although the proofs of the celebrated compositional shuffle conjecture and its rational generalization, along with various symmetric function identities, have largely resolved these conjectures in the affirmative, we provide an elegant alternative proof driven by $q$-series techniques. Through this approach, we derive explicit Schur function expansions for two-column partitions and extend our results to arbitrary positive powers $\nabla^k$.
\end{abstract}
\maketitle

\section{Introduction}
Building upon the foundational work of I. G. Macdonald \cite{Macdonald1988}, the modified Macdonald polynomials $\tilde{H}_{\mu}[X;q,t]$, introduced by Garsia and Haiman \cite{MR1214091}, have emerged as a central object of study in modern algebraic combinatorics. They constitute a remarkable basis of symmetric functions situated at the intersection of algebraic combinatorics, representation theory, and algebraic geometry. In a celebrated result, Haiman \cite{MR1839919} established that these polynomials arise as the Frobenius characteristic of the doubly graded Garsia–Haiman module \cite{MR1214091}. Here, the Frobenius characteristic map serves as a crucial bridge between symmetric functions and the representation theory of the symmetric group $\mathcal{S}_{n}$, mapping the irreducible characters $\chi^{\lambda}$ to the Schur functions $s_{\lambda}$. Haiman's breakthrough ultimately resolved Macdonald's positivity conjecture, which asserts that the Schur expansion of modified Macdonald polynomials yields coefficients in $\mathbb{N}[q,t]$, meaning that they are polynomials in $q$ and $t$ with nonnegative integer coefficients.

A pivotal tool in navigating this rich algebraic structure is the $\nabla$ operator, introduced by Bergeron and Garsia in \cites{MR1726826}. By definition, $\nabla$ is a linear operator that acts diagonally on the modified Macdonald basis via the eigenvalue equation $\nabla \tilde{H}_{\mu}[X;q,t] = T_{\mu}\tilde{H}_{\mu}[X;q,t]$, where $T_{\mu} = t^{n(\mu)}q^{n(\mu')}$ with $n(\mu)=\sum_{i\geq 1}(i-1)\mu_i$. Over the past few decades, understanding the action of the $\nabla$ operator on various symmetric functions has been a major driving force in the field, as evidenced by several foundational works \cites{MR1256101,MR1701592,MR1803316,MR2138143,MR2115257,MR2957232}. A central, overarching theme in this pursuit is the quest for Schur positivity. In the realm of symmetric functions, establishing Schur positivity is of profound importance: it typically implies that the function in question is not merely a formal algebraic sum, but rather the genuine character of a naturally occurring $\mathcal{S}_n$-module, thereby revealing deep underlying geometric or representation-theoretic phenomena. A quintessential example of this phenomenon is the symmetric function $\nabla e_n$, whose Schur positivity follows naturally from its realization as the Frobenius characteristic of the module of diagonal harmonics \cite{MR1256101}. Inspired by these profound connections, the algebraic combinatorics community has extensively investigated the action of the $\nabla$ operator on other fundamental bases. For instance, it is well documented that the symmetric functions $\nabla\omega p_n$ \cites{MR2255191,MR3682738}, $\nabla s_{\lambda}$ \cites{MR2418288,MR4930329,extendingsf}, $\nabla m_{\mu}$ \cites{MR3940644,MR4921577}, $\Delta_{e_k}e_n$ \cites{MR3811519,MR4401822,MR4553915}, and $\Delta_{s_{\lambda}}e_n$ \cite{MR3932967} all exhibit Schur positivity (though some cases remain conjectural), up to a potential global sign. 

Within this context, a natural yet formidable problem is to understand the action of $\nabla$ on the modified Hall--Littlewood polynomials. These are defined by specializing the modified Macdonald polynomials at $q=0$, and their Schur positivity is well known, as they arise as the Frobenius characteristics of the Garsia--Procesi modules \cite{MR1168926}. In an influential paper \cite{MR1803316}, Bergeron, Garsia, Haiman, and Tesler posed the following two conjectures regarding this highly nontrivial action.

\begin{conjecture}\cite{MR1803316}*{Conjecture II}\label{cnj-2}
For any partitions $\lambda$, $\mu$, we have
\begin{align*}
\langle(-1)^{|\mu|-\ell(\mu)}\nabla \tilde{H}_{\mu}[X;0,t],s_{\lambda}\rangle\in \mathbb{N}[q,t].
\end{align*} 
\end{conjecture}

\begin{conjecture}\cite{MR1803316}*{Conjecture III}\label{cnj-3}
For any partition $\mu$, $\nabla \omega\tilde{H}_{\mu}[X;0,t^{-1}]$ is Schur positive. Furthermore, if $\lambda$ and $\mu$ are partitions such that $\mu \unrhd \lambda$ in dominance order, then the difference 
\begin{align*}
\nabla\omega\tilde{H}_{\mu}[X;0,t^{-1}]-\nabla\omega\tilde{H}_{\lambda}[X;0,t^{-1}]    
\end{align*}
is also Schur positive. 
\end{conjecture}

In \cite{MR2957232}, the authors introduced two new families of creation operators, denoted by $\B_a$ and $\C_a$ for $a\in\Z$, and explored their connections to the original creation operators $\H_a$ for modified Hall--Littlewood polynomials $Q_{\mu'}[X;q]$ defined in \cite{MR1112626}. Employing these operators, they proposed a compositional refinement of the shuffle conjecture (originally stated in \cite{MR2115257} and subsequently proven by Carlsson and Mellit \cite{MR3787405}), formulated as: for any composition $\alpha\vDash n$,
\begin{equation}\label{eq:compositional shuffle theorem}
\nabla \C_{\alpha}=\sum_{P\in\LD(n),\comp(P)=\alpha}q^{\dinv(P)}t^{\area(P)}x^P.
\end{equation}
The combinatorial sum on the right-hand side can be naturally expressed as a weighted sum of LLT polynomials \cite{MR1434225}. Relying on the foundational result of Grojnowski and Haiman \cite{grojnowski2007affine} regarding the Schur positivity of all LLT polynomials, this expansion immediately establishes the Schur positivity of $\nabla \C_{\alpha}$. Later, Mellit \cite{MR4348234} proved the rational generalization of the compositional shuffle conjecture, thereby establishing the Schur positivity of $\nabla^k \C_{\alpha}$ for any positive integer $k$.

In \cite{MR2957232}, the authors proved the following symmetric function identities: for any partition $\mu$,
\begin{align}\label{eq:C-alpha identity 1}
(-1)^{|\mu|-\ell(\mu)}\tilde{H}_{\mu}[X;0,t]&=t^{n(\mu)+|\mu|-\ell(\mu)} \C_{\mu}[X;t],\\ 
\omega \tilde{H}_{\mu}[X;0,t^{-1}]&=\sum_{\alpha\leq \overleftarrow{\mu}}t^{\doff_{\overleftarrow{\mu}}(\DP(\alpha))-n(\mu)}\C_{\alpha}[X;t], \label{eq:C-alpha identity 2}
\end{align}
where $\overleftarrow{\mu}:=(\mu_{\ell},\ldots,\mu_1)$ denotes the reversal of $\mu$, $\DP(\alpha)$ represents a Dyck path associated with $\alpha$, and $\doff$ is a specific statistic. 

Consequently, combining these identities with the aforementioned results, the Schur positivity of modified Hall--Littlewood polynomials under the action of $\nabla^k$ naturally follows. Thus, the only remaining piece of the above conjectures is the Schur positivity of their differences under the dominance order.

In this article, we introduce an alternative framework to study the aforementioned conjectures, aiming to resolve them without relying on the compositional shuffle theorem. The fundamental difficulty in proving these conjectures stems from the fact that, while $\nabla$ acts by scalar multiplication on the modified Macdonald polynomials via $\nabla\tilde{H}_{\mu}[X;q,t]=T_{\mu}\tilde{H}_{\mu}[X;q,t]$, the modified Hall–Littlewood polynomials $\tilde{H}_{\mu}[X;0,t]$ are no longer its eigenfunctions. Since obtaining an explicit expansion of $\tilde{H}_{\mu}[X;0,t]$ in the modified Macdonald basis remains intractable, the image of these functions under the $\nabla$ operator is highly elusive, making their exact Schur positive expansions extremely difficult to predict. We overcome this obstacle for the two-column case $\mu =(2^a, 1^b)$. By establishing explicit algebraic identities that evaluate the action of $\nabla$, we provide an alternative proof of Conjectures \ref{cnj-2} and \ref{cnj-3} for these partitions and extend these results to arbitrary powers $\nabla^k$, as summarized in our main theorem below.

\begin{maintheorem}
Let $a$ and $b$ be nonnegative integers, and let $k$ be a positive integer. Define a lower-triangular transition matrix $M=(M_{i,j})_{0\leq i,j\leq a}$, where the rows and columns are indexed from $0$ to $a$, with entries given by
\begin{align*}
M_{i,j}=T_{2^{a-i}1^{b+2i}}\tbinom{a-j}{i-j}\in \mathbb{N}[q,t].     
\end{align*}
Let $M_{i,j}^{(k)}$ denote the entry of the matrix power $M^k$ corresponding to the indices $i$ and $j$. This entry inherently belongs to $\mathbb{N}[q,t]$ as well. Then, the following identities hold: 
\begin{enumerate}
\item[(i)] For any integer $j$ such that $0\leq j\leq a$,
\begin{align*}
\nabla^{k}\omega\tilde{H}_{2^{a-j}1^{b+2j}}[X;0,t^{-1}]=\sum_{i=j}^{a} M_{i,j}^{(k)}\omega\tilde{H}_{2^{a-i}1^{b+2i}}[X;0,t^{-1}].
\end{align*}
\item[(ii)] We have
\begin{align*}
(-1)^{a}\nabla^{k}\tilde{H}_{2^a1^b}[X;0,t]=\sum_{i=0}^{a} q^{a}t^{\binom{2a+b}{2}+\binom{a+b}{2}}M_{i,0}^{(k-1)}\omega\tilde{H}_{2^{a-i}1^{b+2i}}[X;0,t^{-1}].
\end{align*}
\end{enumerate}
\end{maintheorem}

As an immediate consequence of these explicit identities and the established Schur positivity of the transformed modified Hall--Littlewood polynomials $\omega\tilde{H}_{\mu}[X;0,t^{-1}]$, we affirmatively verify Conjectures \ref{cnj-2} and \ref{cnj-3} in the two-column case. A detailed discussion of these results is deferred to Section 6.

We are grateful to an anonymous expert for pointing out that the explicit transition matrix $M^{(k)}$ established in our Main Theorem can also be deduced via the compositional shuffle conjecture and LLT polynomial theory. In particular, any Dyck path with a touch composition that rearranges to $(2^a,1^b)$ is inherently balanced \cite{MR2371044}*{Theorem 6.8}. As a result, the expansion of $\nabla \C_{2^a1^b}$ simplifies to a scalar multiple of the single modified Hall--Littlewood polynomial $\omega \tilde{H}_{2^a1^b}[X;0,t^{-1}]$. By analyzing the minimal $\dinv$ and $\doff$ statistics in~\eqref{eq:compositional shuffle theorem},\eqref{eq:C-alpha identity 1}, and \eqref{eq:C-alpha identity 2} across all refinements of these balanced Dyck paths, the explicit $q,t$-binomial coefficients comprising the matrix $M$ can, in principle, be computed. While we refrain from detailing the extensive combinatorial definitions and analysis required for this approach, we emphasize that, to the best of our knowledge, such explicit closed-form expressions and their matrix formulation have not previously appeared in the literature.

Throughout this work, our central methodology relies on $q$-series techniques. However, for readers familiar with the underlying combinatorial machinery, Appendix A provides an alternative proof of one of our main symmetric function identities via the framework of Dyck path algebras. The remainder of this paper is organized as follows. Section 2 provides a brief background on symmetric functions and $q$-series, summarizing the essential lemmas utilized in our subsequent proofs. Section 3 derives the modified Macdonald expansion for the two-column modified Hall--Littlewood polynomials. In Section 4, we evaluate their image under the $\nabla$ operator, thereby providing an alternative proof of Conjecture \ref{cnj-2} in this special case. Section 5 employs this evaluation as a central tool to both re-establish the first part and resolve the second part of Conjecture \ref{cnj-3} in the two-column case, and we further generalize our findings to arbitrary powers of the $\nabla$ operator in Section 6.

\section{Preliminary}

In this section, we briefly recall standard notation and preliminary concepts regarding integer partitions, Young tableaux, symmetric functions, and $q$-series. We restrict our exposition to the essentials required for our subsequent proofs. We refer the reader to \cites{MR1354144, MR1676282, MR2371044, MR2538310} for a comprehensive introduction to the theory of symmetric functions and Macdonald polynomials, and to \cite{MR2128719} for an in-depth overview of $q$-series.

\subsection{Partitions and tableaux}
A \emph{partition} $\lambda$ is a finite nonincreasing sequence $\lambda_{1}\geq\lambda_{2}\geq\cdots \geq\lambda_{\ell}>0$ of positive integers. Each $\lambda_{i}$ is called the $i$-th \emph{part} of $\lambda$. The number of parts, denoted by $\ell(\lambda)$, is called the \emph{length} of $\lambda$, and the sum of its parts, $|\lambda|:=\sum_{i=1}^{\ell}\lambda_{i}$, is called the \emph{size} of $\lambda$. We write $\lambda\vdash n$ to denote that $\lambda$ is a partition of size $n$.

For any two partitions $\lambda$ and $\mu$ (extended by trailing zeros if necessary), we define $\lambda \subseteq\mu$ if $\lambda_i\leq\mu_i$ for all $i\geq 1$. Furthermore, given two partitions of the same integer $n$, denoted $\lambda,\mu\vdash n$, we say that $\lambda$ \emph{dominates} $\mu$, written $\lambda\unrhd\mu$, if for all $k\geq 1$, $\sum_{i=1}^k \lambda_i\geq\sum_{i=1}^k\mu_i$.

The \emph{Young diagram} associated with a partition $\lambda$ is a left-aligned array of unit squares, referred to as \emph{cells}, with $\lambda_i$ cells in the $i$-th row. By convention, we often use $\lambda$ to denote both the partition and its diagram. The \emph{conjugate partition}, denoted by $\lambda'$, is defined as the partition whose Young diagram is obtained from that of $\lambda$ via reflection across the main diagonal. In this article, we adopt the French notation for these diagrams.

Given a partition $\mu$ and a cell $c\in\mu$, let the $\mathrm{arm}_{\mu}(c)$, $\mathrm{leg}_{\mu}(c)$, $\mathrm{coarm}_{\mu}(c)$, and $\mathrm{coleg}_{\mu}(c)$ denote the number of cells strictly between $c$ and the boundary of $\mu$ in the East, North, West, and South directions, respectively. We now introduce several standard partition statistics that will be used throughout this work. First, define $n(\mu):=\sum_{i=1}^{\ell}(i-1)\mu_i$. We then define the polynomials:
\begin{align}
T_{\mu}:=q^{n(\mu')}t^{n(\mu)}=\prod_{c\in \mu}q^{\mathrm{coarm}_{\mu}(c)}t^{\mathrm{coleg}_{\mu}(c)}\quad  \mathrm{and}\quad B_{\mu}:=\sum_{c\in \mu}q^{\mathrm{coarm}_{\mu}(c)}t^{\mathrm{coleg}_{\mu}(c)}.
\end{align}

\begin{example}
The following diagrams illustrate the partition $\lambda=(3,3,1)$, its conjugate $\lambda'=(3,2,2)$, and the pair $(\mathrm{coarm}, \mathrm{coleg})$ for the cells in $\lambda$. Note the containment $(3,2)\subseteq(3,3,1)$ and the dominance $(3,3,1)\unrhd (3,2,2)$. Furthermore, we have $n((3,3,1))=5$ and $n((3,2,2))=6$, which gives $T_{(3,3,1)}(q,t)=q^6t^5$. Additionally, evaluating the sum over the cells directly yields $B_{(3,3,1)}(q,t)= q^2t+q^2+qt+t^2+q+t+1$.
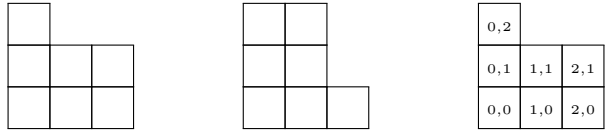
\begin{align*}
\ydiagram{1,3,3} \qquad \qquad \ydiagram{2,2,3} \qquad \qquad \begin{ytableau}
\scriptscriptstyle 0,2 \\
\scriptscriptstyle 0,1 & \scriptscriptstyle 1,1 & \scriptscriptstyle 2,1 \\
\scriptscriptstyle 0,0 & \scriptscriptstyle 1,0 & \scriptscriptstyle 2,0
\end{ytableau}
\end{align*}
\end{example}

Given partitions $\lambda, \mu \vdash n$, a \emph{semistandard Young tableau} (SSYT) of \emph{shape} $\lambda$ and \emph{weight} $\mu$ is defined as a filling of the Ferrers diagram of $\lambda$ using elements from the multiset $\{1^{\mu_1}, 2^{\mu_2}, \dots\}$ such that the entries weakly increase across each row and strictly increase up each column. Let $\mathrm{SSYT}(\lambda, \mu)$ denote the set of these fillings, and let $K_{\lambda, \mu}$ be its cardinality. The integer $K_{\lambda, \mu}$ is known as the \emph{Kostka number}.

Among the various combinatorial statistics on SSYTs, the \emph{cocharge} statistic, a modification of the \emph{charge} statistic first introduced in \cite{MR472993}, is central to our study. We follow the notations in \cite{MR4775674}.

\begin{definition}
Given $T\in \mathrm{SSYT}(\lambda,\mu)$, the \emph{reading word} $\sigma(T)$ is obtained by reading the entries row by row from top to bottom, and from left to right within each row. The first \emph{cocharge subword} is obtained by searching the reading word from right to left for a $1$, then continuing the search from that position to locate a $2$ (cyclically wrapping around the word as needed), and proceeding in this manner up to the maximal entry of the word. The \emph{cocharge labeling} of a permutation is determined by performing a similar right-to-left cyclic search. Processing the entries $1, 2, \dots, n$ in increasing order, we assign a label of $0$ to the entry $1$. The label assigned to $i+1$ is incremented if and only if $i+1$ appears to the left of $i$ in the permutation; otherwise, it remains the same. We then similarly find and label the second cocharge subword among the remaining unlabeled letters. This process is repeated iteratively until all letters have been assigned a label.
The \emph{cocharge} of $T$ is the sum of the cocharge labels of its reading word.
\end{definition}

\begin{example}
In the following example, we have $T\in\mathrm{SSYT}_{(3,3,1),(3,2,1,1)}$ and $\sigma(T)=3224111$. The three cocharge subwords are $3_{2}2_{1}4_{2}1_{0}$, $2_{1}1_{0}$, and $1_0$, where the subscript denotes the corresponding cocharge label. Thus we have $\mathrm{cocharge}(T)=2+1+2+1=6$.
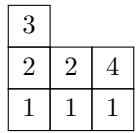
\begin{align*}
\begin{ytableau}
3 \\
2 & 2 & 4\\
1 & 1 & 1
\end{ytableau}   
\end{align*}
\end{example}

The \emph{modified Kostka--Foulkes polynomial} is defined as
\begin{align}\label{eq-kostka-foulkes}
\tilde{K}_{\lambda,\mu}(t):=\sum_{T\in \mathrm{SSYT}(\lambda,\mu)}t^{\mathrm{cocharge}(T)}.
\end{align}
For an extensive survey on Kostka--Foulkes polynomials, we refer the reader to \cite{MR1865038}.

\subsection{Symmetric functions}

Let $\Lambda$ denote the ring of symmetric functions over a base field $\mathbb{F}$. Algebraically, $\Lambda$ can be viewed as the polynomial ring freely generated by the power sum symmetric functions $p_1, p_2, \dots$, so that $\Lambda = \mathbb{F}[p_1, p_2, \dots]$. It is equipped with a natural grading $\Lambda = \bigoplus_{n \geq 0} \Lambda^{(n)}$, where $\Lambda^{(n)}$ consists of homogeneous symmetric functions of degree $n$, defined by assigning $\deg(p_k) = k$. In the context of Macdonald polynomials, we will work over the field of rational functions $\mathbb{Q}(q,t)$. Additionally, we will adopt the notation $F[X] := F(x_1, x_2, \dots)$ to denote a symmetric function evaluated on the variable set $X$.

The monomial $\{m_{\lambda}\}_{\lambda\vdash n}$, elementary $\{e_{\lambda}\}_{\lambda\vdash n}$, complete homogeneous $\{h_{\lambda}\}_{\lambda\vdash n}$, power sum $\{p_{\lambda}\}_{\lambda\vdash n}$, and Schur $\{s_{\lambda}\}_{\lambda\vdash n}$ functions constitute five classical bases for the space $\Lambda^{(n)}$. 

Schur functions play a central role in the theory of symmetric functions. Recall that the combinatorial definition of the Schur function is given by
\begin{align*}
s_{\lambda}:=\sum_{T\in \mathrm{SSYT}(\lambda)}x^{T}=\sum_{\mu\vdash n}K_{\lambda,\mu}m_{\mu}.
\end{align*}
The fundamental \emph{involution} $\omega:\Lambda \rightarrow \Lambda$ acts on the Schur basis via
\begin{align}
\omega(s_{\lambda})=s_{\lambda'}.
\end{align}
Furthermore, the Schur functions form an orthonormal basis with respect to the \emph{Hall scalar product}, satisfying $\langle s_{\lambda},s_{\mu}\rangle=\delta_{\lambda,\mu}$, where $\delta_{\lambda,\mu}$ denotes the Kronecker delta. In our context, a symmetric function $f$ is said to be \emph{Schur positive} if its expansion in the Schur basis yields coefficients in $\mathbb{N}[q,t]$.

Next, we consider two important extensions of the Schur functions: the Hall--Littlewood and Macdonald polynomials, which serve as one-parameter ($t$) and two-parameter ($q,t$) generalizations, respectively. For $\mu\vdash n$, we denote by
\begin{align*}
\tilde{H}_{\mu}[X;q,t]=\sum\limits_{\lambda\vdash n}\tilde{K}_{\lambda,\mu}(q,t)s_{\lambda},
\end{align*}
the \emph{modified Macdonald polynomial}, where the coefficients
\begin{align}\label{eq-qtKostka}
\tilde{K}_{\lambda,\mu}(q,t)=t^{n(\mu)}K_{\lambda,\mu}(q,t^{-1}),
\end{align}
are the \emph{modified $q,t$-Kostka polynomials}. The family $\{\tilde{H}_{\mu}[X;q,t]\}_{\mu\vdash n}$ also forms a basis for $\Lambda^{(n)}$. 

Let $\nabla$ be the linear operator on symmetric functions which satisfies
\begin{align*}
\nabla \tilde{H}_{\mu}[X;q,t]=T_{\mu}\tilde{H}_{\mu}[X;q,t].    
\end{align*}
This remarkable operator has played an important role in the theory of diagonal harmonics over the past three decades.

While an explicit closed formula for the general modified $q,t$-Kostka polynomials remains elusive, specializing at $q=0$ recovers the modified Kostka--Foulkes polynomials in the parameter $t$, whose combinatorial definition is provided in \eqref{eq-kostka-foulkes}. Correspondingly, this specialization yields the \emph{modified Hall--Littlewood polynomials}
\begin{align*}
\tilde{H}_{\mu}[X;0,t]=\sum_{\lambda\vdash n}\tilde{K}_{\lambda,\mu}(t)s_{\lambda}.
\end{align*}

We close this subsection with a few symmetric function identities required for the proof. First, it is a standard fact that the modified Macdonald polynomial indexed by a single column is independent of $q$.
\begin{lemma}
We have
\begin{align}\label{eq-Ht1^n}
\tilde{H}_{1^n}[X;q,t]=\tilde{H}_{1^n}[X;0,t]=(t;t)_{n}h_{n}\left[\frac{X}{1-t}\right].
\end{align}
\end{lemma}

The symmetry relation for modified $(q,t)$-Kostka polynomials, which was originally explored by Macdonald in the classical setting \cite{MR1354144} (cf. \cite{MR2371044}*{Section 2}), states that
\begin{align*}
\tilde{K}_{\lambda,\mu}(q,t)=q^{n(\mu')}t^{n(\mu)}\tilde{K}_{\lambda',\mu}(q^{-1},t^{-1}).
\end{align*}
The symmetry implies the following well-known property.

\begin{lemma}
We have
\begin{align}\label{equ-HtwHt}
\tilde{H}_{\mu}[X;q,t]=q^{n(\mu')}t^{n(\mu)}\omega\tilde{H}_{\mu}[X;q^{-1},t^{-1}].
\end{align}
\end{lemma}

The following symmetry property is implicitly contained in \cite{MR1701592}*{[Theorem I.1]} and appears explicitly in the proof of Theorem 4.1 in \cite{MR4547238}.

\begin{lemma}\label{lem-Hlambda=Hmu}
For any two partitions $\lambda$ and $\mu$ of $n$, and any integers $r$ and $s$, 
\begin{align*}
B_{\lambda}(t^r,t^s) = B_{\mu}(t^r,t^s) \implies \tilde{H}_{\lambda}[X;t^r,t^s]=\tilde{H}_{\mu}[X;t^r,t^s].
\end{align*}
\end{lemma}

\subsection{\texorpdfstring{$t$}{t}-Pochhammer symbol}
We now introduce the basic notation for $q$-series. For consistency with our subsequent formulation, we adopt $t$ as the primary parameter instead of the conventional $q$.

For integers $n\geq 0$ and indeterminates $a$ and $t$, we define the \emph{$t$-shifted factorial} by
\begin{align*}
(a;t)_n:=(1-a)(1-at)\cdots(1-at^{n-1}).
\end{align*}
The \emph{$t$-binomial coefficient} is then defined for $0 \leq k \leq n$ as
\begin{align}\label{equ-t-binomial}
\begin{bmatrix}
n\\k
\end{bmatrix}_t:=\frac{(t;t)_{n}}{(t;t)_k(t;t)_{n-k}}.
\end{align}
By convention, we set $\tbinom{n}{k}:=0$ if $k<0$ or $k>n$. Although not immediately apparent from its algebraic definition, we have the following well-known combinatorial interpretation.

\begin{lemma}
For any integers $n \ge k \ge 0$, we have
\begin{equation}
\tbinom{n}{k}=\sum_{\lambda\subseteq (n-k)^{k}}t^{|\lambda|}.
\end{equation}
Consequently, it is a polynomial in $\mathbb{N}[t]$.
\end{lemma}

To facilitate the algebraic manipulations in the subsequent proofs, we recall several elementary identities involving ordinary binomial coefficients and $t$-shifted factorials. Their proofs are straightforward and thus omitted.

\begin{lemma}\label{lem-A+-B}
For any two integers $a$ and $b$, we have
\begin{align*}
\binom{a+b}{2}=\binom{a}{2}+\binom{b}{2}+ab \quad  \text{ and } \quad \binom{a-b}{2}=\binom{a}{2}+\binom{b+1}{2}-ab.
\end{align*}
\end{lemma}

\begin{lemma}
For any nonnegative integers $n, m, N$ and indeterminates $x, t$, we have:

\medskip
\noindent\emph{Inversion formula:}
\begin{align}\label{eq-inversion}
(x;t)_n=(-1)^n x^n t^{\binom{n}{2}}(x^{-1}t^{1-n};t)_n.
\end{align}
\noindent\emph{Splitting formula:}
\begin{align}\label{eq-splitting}
    (x;t)_{n+m}=(x;t)_n(xt^n;t)_m.
\end{align}
\noindent\emph{Quotient formulas:} For $N\ge n\ge 0$ and integer $A\ge 1$,
\begin{align}
    (t^{-N};t)_n&=(-1)^n t^{-Nn+\binom{n}{2}}\frac{(t;t)_N}{(t;t)_{N-n}}, \label{eq-quotient-neg} \\
    (t^A;t)_n&=\frac{(t;t)_{A+n-1}}{(t;t)_{A-1}}. \label{eq-quotient-pos}
\end{align}
\noindent\emph{Punctured product:} For $0 \le k \le n-1$,
\begin{align}\label{eq-punctured}
    \prod_{j=0,\,j\neq k}^{n-1}(1-xt^j) = (x;t)_k (xt^{k+1};t)_{n-1-k}.
\end{align}
\end{lemma}

Furthermore, our analysis relies on the following three well-known identities. The first is a standard exchange relation.
\begin{lemma}\label{lem-tbinomexchange}
Let $n$, $k$, and $m$ be nonnegative integers. Then we have
\begin{align}
\tbinom{n}{k}\tbinom{k}{m}=\tbinom{n}{m}\tbinom{n-m}{n-k}.
\end{align}
\end{lemma}

The second is a fundamental result in the theory of basic hypergeometric series.

\begin{lemma}[$t$-Chu--Vandermonde identity]\label{lem-chuvandermonde}
For any integer $n \geq 0$ and indeterminates $b$ and $c$, the following terminating basic hypergeometric identity holds:
\begin{align}
\sum_{k=0}^{n}\frac{(t^{-n};t)_k(b;t)_k}{(c;t)_k(t;t)_k}t^k= \frac{(c/b;t)_n}{(c;t)_n}b^n.
\end{align}
\end{lemma}

Finally, the third identity can be found in \cite{MR1858275}.
\begin{lemma}[Inverse $t$-binomial expansion]\label{lem-inversetbinomial}
For any indeterminate $z$ and nonnegative integer $n$, the following identity holds:
\begin{align}
z^{n}=\sum_{k=0}^{n}(-1)^{k}t^{\binom{k+1}{2}-nk}\tbinom{n}{k}(z;t)_{k}.
\end{align}
\end{lemma}

\section{Expansion in terms of the modified Macdonald polynomials}

In \cite{MR1182707}, the author established an explicit formula for the classical two-column Kostka--Foulkes polynomials. By applying \eqref{eq-qtKostka}, we obtain the following expansion for the two-column modified Macdonald polynomials in terms of the modified Hall--Littlewood basis.

\begin{lemma}\cite{MR1182707}*{Reformulation of Theorem 1.1}
For any nonnegative integers $a$ and $b$,
\begin{align}\label{eq-Mac2HL}
\tilde{H}_{2^a1^b}[X;q,t]=\sum_{r=0}^{a}q^{a-r}t^{-(a-r)(a+b-r)}(qt^{-(a+b)};t)_{r}\frac{(t^{-a};t)_r}{(t^{-r};t)_r}\tilde{H}_{2^r1^{2a+b-2r}}[X;0,t].
\end{align}
\end{lemma}

For related variants of the above expansion, we refer the reader to \cites{MR1663580,MR2948763}. Motivated by computational experiments, we now establish the corresponding inverse expansion.
\begin{lemma}
For any nonnegative integers $a$ and $b$, we have
\begin{align}\label{eq-HL2Mac}
\tilde{H}_{2^a1^b}[X;0,t]=\sum_{m=0}^{a}\frac{(-q)^{-m}t^{\binom{a+b+1}{2}-\binom{a+b+1-m}{2}}\tbinom{a}{m}}{(q^{-1}t^{2a+b-2m+1}; t)_{m}(q^{-1}t^{a+b-m};t)_{a-m}}\tilde{H}_{2^{m}1^{2a+b-2m}}[X;q,t].
\end{align}
\end{lemma}
\begin{proof}
We proceed by induction on $a$. When $a=0$, the partition reduces to $(1^b)$. The desired result then follows directly from \eqref{eq-Ht1^n}, which establishes the base case.

For the inductive step, we now assume $a\geq 1$ and that the result holds for all smaller values of $a$ (and for any non-negative integer $b$). By \eqref{eq-Mac2HL}, we can write
\begin{align*}
\tilde{H}_{2^a1^b}[X;q,t]=\sum_{r=0}^{a}C_{a,r}\tilde{H}_{2^r1^{2a+b-2r}}[X;0,t],
\end{align*}
where
\begin{align*}
C_{a,r}:=q^{a-r}t^{-(a-r)(a+b-r)}(qt^{-(a+b)};t)_{r}\frac{(t^{-a};t)_r}{(t^{-r};t)_r}.
\end{align*}
Isolating the $r=a$ term and noting that $C_{a,a}=(qt^{-(a+b)};t)_a$, we have
\begin{align}\label{eq-HL2a1b-C}
\tilde{H}_{2^a1^b}[X;0,t]=\frac{1}{C_{a,a}}\tilde{H}_{2^a1^b}[X;q,t]-\sum_{r=0}^{a-1}\frac{C_{a,r}}{C_{a,a}}\tilde{H}_{2^r1^{2a+b-2r}}[X;0,t].
\end{align}
We can apply the induction hypothesis to each partition $(2^r,1^{2a+b-2r})$. This yields
\begin{align*}
\tilde{H}_{2^r1^{2a+b-2r}}[X;0,t]=\sum_{m=0}^{r}D_{r,m}\tilde{H}_{2^m1^{2a+b-2m}}[X;q,t],
\end{align*}
where
\begin{align*}
D_{r,m}:=\frac{(-q)^{-m}t^{\binom{2a+b-r+1}{2}-\binom{2a+b-r+1-m}{2}}\tbinom{r}{m}}{(q^{-1}t^{2a+b-2m+1}; t)_{m}(q^{-1}t^{2a+b-r-m};t)_{r-m}}.
\end{align*}
Substituting this identity into \eqref{eq-HL2a1b-C}, we obtain
\begin{align*}
\tilde{H}_{2^a1^b}[X;0,t] &= \frac{1}{C_{a,a}}\tilde{H}_{2^a1^b}[X;q,t]-\sum_{r=0}^{a-1}\frac{C_{a,r}}{C_{a,a}}\sum_{m=0}^{r}D_{r,m}\tilde{H}_{2^m1^{2a+b-2m}}[X;q,t] \\ &= \frac{1}{C_{a,a}}\tilde{H}_{2^a1^b}[X;q,t]-\sum_{m=0}^{a-1}\sum_{r=m}^{a-1}\frac{C_{a,r}}{C_{a,a}}D_{r,m}\tilde{H}_{2^m1^{2a+b-2m}}[X;q,t].
\end{align*}
To complete the induction step, we must verify that
\begin{align*}
\tilde{H}_{2^a1^b}[X;0,t] &= \sum_{m=0}^{a}D_{a,m}\tilde{H}_{2^m1^{2a+b-2m}}[X;q,t] \\ &= D_{a,a}\tilde{H}_{2^a1^b}[X;q,t]+\sum_{m=0}^{a-1}D_{a,m}\tilde{H}_{2^m1^{2a+b-2m}}[X;q,t].
\end{align*}
Comparing the coefficients, it suffices to show that
\begin{align*}
C_{a,a}D_{a,a}=1 \quad \text{and} \quad \sum_{r=m}^{a}C_{a,r}D_{r,m}=0 \quad \text{for } 0\leq m<a.
\end{align*}
We will prove these two identities in Propositions \ref{prop-caadaa} and \ref{prop-sumcardrm} below. Assuming these results for the moment, the induction step is complete.
\end{proof}

\begin{proposition}\label{prop-caadaa}
With $C_{a,r}$ and $D_{r,m}$ defined as above, then $C_{a,a}D_{a,a}=1$.
\end{proposition}
\begin{proof}
Recall that $C_{a,a}=(qt^{-(a+b)};t)_a$. By definition, we have
\begin{align*}
D_{a,a}=(-1)^{a}\frac{q^{-a}t^{\binom{a+b+1}{2}-\binom{b+1}{2}}}{(q^{-1}t^{b+1};t)_a}.
\end{align*}
Applying \eqref{eq-inversion} with $x=q^{-1}t^{b+1}$ and $n=a$, the denominator can be rewritten directly as
\begin{align*}
(q^{-1}t^{b+1};t)_a&=(-1)^{a}(q^{-1}t^{b+1})^{a}t^{\binom{a}{2}}(qt^{-b-1}t^{1-a};t)_a\\
&=(-1)^{a}q^{-a}t^{\binom{a+b+1}{2}-\binom{b+1}{2}}(qt^{-(a+b)};t)_a
\end{align*}
Substituting this back into the expression for $D_{a,a}$, it immediately follows that $C_{a,a}D_{a,a}=1$.
\end{proof}

\begin{proposition}\label{prop-sumcardrm}
The identity $\sum_{r=m}^{a}C_{a,r}D_{r,m}=0$ holds for any integer $m$ such that $0\leq m<a$, where $C_{a,r}$ and $D_{r,m}$ are defined as above.
\end{proposition}

\begin{proof}
Let us evaluate the following summation:
\begin{align*}
\sum_{r=m}^{a}q^{a-r}t^{-(a-r)(a+b-r)}(qt^{-(a+b)};t)_{r}\frac{(t^{-a};t)_r}{(t^{-r};t)_r}\frac{(-q)^{-m}t^{\binom{2a+b-r+1}{2}-\binom{2a+b-r+1-m}{2}}\tbinom{r}{m}}{(q^{-1}t^{2a+b-2m+1}; t)_{m}(q^{-1}t^{2a+b-r-m};t)_{r-m}}.
\end{align*}

By \eqref{eq-splitting}, we first extract the $m$-dependent terms in the numerator:
\begin{align*}
(qt^{-(a+b)};t)_r&=(qt^{-(a+b)};t)_m(qt^{-(a+b)+m};t)_{r-m},\\
(t^{-a};t)_r&=(t^{-a};t)_m(t^{-a+m};t)_{r-m}.
\end{align*}
Additionally, by expanding the $t$-binomial coefficient and applying \eqref{eq-quotient-neg} to the denominator, we obtain
\begin{align*}
\frac{\tbinom{r}{m}}{(t^{-r};t)_r}=\frac{(-1)^rt^{\binom{r+1}{2}}}{(t;t)_m(t;t)_{r-m}}.   
\end{align*}
For the denominator term, applying \eqref{eq-inversion} yields
\begin{align*}
(q^{-1}t^{2a+b-r-m};t)_{r-m}&=(-1)^{r-m}(q^{-1}t^{2a+b-r-m})^{r-m}t^{\binom{r-m}{2}} (qt^{2m+1-2a-b};t)_{r-m} \\
&=(-q^{-1}t^{2a+b-2m})^{r-m}t^{-\binom{r-m+1}{2}}(qt^{2m+1-2a-b};t)_{r-m}.
\end{align*}
Substituting these simplified expressions back into our summation, we can explicitly track the exact cancellations for the signs, powers of $q$, and powers of $t$:
\begin{align*}
(-1)^{m}(-1)^{r}(-1)^{r-m}&=1, \\
q^{a-r}q^{-m}q^{r-m}&=q^{a-2m}, \\
t^{-(a-r)(a+b-r)+\binom{2a+b-r+1}{2}-\binom{2a+b-r+1-m}{2}+\binom{r+1}{2}+\binom{r-m+1}{2}+(2a+b-2m)(m-r)}&=t^{\gamma+r-m},
\end{align*}
where $\gamma$ is a constant exponent depending only on $a, b$, and $m$. By gathering $t^\gamma$, $q^{a-2m}$, and all other factors independent of $r$ into a single overall coefficient $K$, the summation cleanly reduces to:
\begin{align*}
\sum_{r=m}^{a}C_{a,r}D_{r,m}&=K\sum_{r=m}^{a}\frac{(qt^{-(a+b)+m};t)_{r-m}(t^{-a+m};t)_{r-m}}{(qt^{2m+1-2a-b};t)_{r-m}(t;t)_{r-m}}t^{r-m}\\
&=K\sum_{k=0}^{a-m}\frac{(t^{-(a-m)};t)_k(qt^{-(a+b)+m};t)_{k}}{(qt^{2m+1-2a-b};t)_k(t;t)_k}t^k \\
&=K\cdot\frac{(t^{m+1-a};t)_{a-m}}{(qt^{2m+1-2a-b};t)_{a-m}}(qt^{-(a+b)+m})^{a-m}.
\end{align*}
Note that the second equality follows from the index shift $k=r-m$, while the final equality is an application of the $t$-Chu--Vandermonde identity in Lemma \ref{lem-chuvandermonde}. 

Finally, expanding the numerator of the resulting expression yields
\begin{equation*}
(t^{m+1-a};t)_{a-m} = (1-t^{m+1-a})\cdots(1-t^{m+1-a}t^{a-m-1}).
\end{equation*}
Since $a>m$, the final term evaluates exactly to $(1-t^0) = 0$. Therefore, the entire sum vanishes.
\end{proof}

\section{An alternative proof of Conjecture \ref{cnj-2} in the two-column case}

In this section, we establish the following theorem, which gives an alternative proof of Conjecture \ref{cnj-2} for two-column partitions. Furthermore, we provide an explicit formula for the Schur expansion.

\begin{theorem}\label{thm-nablaHt}
For any nonnegative integers $a$ and $b$,
\begin{align*}
(-1)^{a}\nabla \tilde{H}_{2^a1^b}[X;0,t]&=q^{a}t^{\binom{2a+b}{2}+\binom{a+b}{2}}\omega \tilde{H}_{2^a1^b}\left[X;0,t^{-1}\right]\\
&=q^{a}t^{\binom{2a+b}{2}+\binom{a+b}{2}}\sum_{\lambda\vdash 2a+b}\sum_{T\in \mathrm{SSYT}(\lambda,2^a1^b)}t^{-\mathrm{cocharge}(T)}s_{\lambda'}.
\end{align*}
In particular, $(-1)^{a}\nabla \tilde{H}_{2^a1^b}[X;0,t]$ is Schur positive.
\end{theorem}

\begin{proof}
Multiplying both sides of \eqref{eq-HL2Mac} by $(-1)^a$ and applying the $\nabla$ operator, a direct calculation yields the following expression. Letting $\mathrm{LHS}$ denote the left-hand side of our first target identity, we have
\begin{align*}
\mathrm{LHS}&=\sum_{m=0}^{a}\frac{(-1)^{a-m}q^{-m}t^{\binom{a+b+1}{2}-\binom{a+b+1-m}{2}}\tbinom{a}{m}}{(q^{-1}t^{2a+b-2m+1}; t)_{m}(q^{-1}t^{a+b-m};t)_{a-m}}\left( q^{m}t^{\binom{2a+b-m}{2}+\binom{m}{2}}\tilde{H}_{2^m1^{2a+b-2m}}[X;q,t] \right) \\
&=\sum_{m=0}^{a}\frac{(-1)^{a-m}t^{\binom{a+b+1}{2}-\binom{a+b+1-m}{2}+\binom{2a+b-m}{2}+\binom{m}{2}}\tbinom{a}{m}}{(q^{-1}t^{2a+b-2m+1}; t)_{m}(q^{-1}t^{a+b-m};t)_{a-m}}\tilde{H}_{2^m1^{2a+b-2m}}[X;q,t].
\end{align*}

Next, we simplify the exponent of $t$. Applying Lemma \ref{lem-A+-B}, one can readily verify the identity
\begin{align*}
t^{\binom{a+b+1}{2}-\binom{a+b+1-m}{2}+\binom{2a+b-m}{2}+\binom{m}{2}}=t^{\binom{2a+b}{2}+\binom{a-m}{2}-\binom{a}{2}}.    
\end{align*}
Substituting this relation back, we extract the common factor $t^{\binom{2a+b}{2}}$. Simultaneously multiplying and dividing the expression by $q^a$, and reversing the summation index via $r=a-m$, we transform the expression into
\begin{align*}
\mathrm{LHS}&=t^{\binom{2a+b}{2}}\sum_{m=0}^{a}\frac{(-1)^{a-m}t^{\binom{a-m}{2}-\binom{a}{2}}\tbinom{a}{m}}{(q^{-1}t^{2a+b-2m+1};t)_{m}(q^{-1}t^{a+b-m};t)_{a-m}}\tilde{H}_{2^m 1^{2a+b-2m}}[X;q,t]\\
&=q^{a}t^{\binom{2a+b}{2}}\sum_{m=0}^{a}\frac{q^{-a}(-1)^{a-m}t^{\binom{a-m}{2}-\binom{a}{2}}\tbinom{a}{m}}{(q^{-1}t^{2a+b-2m+1};t)_{m}(q^{-1}t^{a+b-m};t)_{a-m}}\tilde{H}_{2^m 1^{2a+b-2m}}[X;q,t]\\
&=q^{a}t^{\binom{2a+b}{2}}\sum_{r=0}^{a}\frac{q^{-a}(-1)^{r}t^{\binom{r}{2}-\binom{a}{2}}\tbinom{a}{r}}{(q^{-1}t^{b+r};t)_{r}(q^{-1}t^{b+2r+1};t)_{a-r}}\tilde{H}_{2^{a-r} 1^{b+2r}}[X;q,t].
\end{align*}

Finally, applying the symmetry relation \eqref{equ-HtwHt} introduces the involution $\omega$:
\begin{align*}
\tilde{H}_{2^{a-r} 1^{b+2r}}[X;q,t]=q^{a-r}t^{\binom{a+b+r}{2}+\binom{a-r}{2}}\omega\tilde{H}_{2^{a-r} 1^{b+2r}}[X;q^{-1},t^{-1}].
\end{align*}
Substituting this back, we obtain
\begin{align*}
\mathrm{LHS}=q^{a}t^{\binom{2a+b}{2}}\sum_{r=0}^{a}\frac{(-1)^{r}t^{\binom{r}{2}-\binom{a}{2}+\binom{a+b+r}{2}+\binom{a-r}{2}}\tbinom{a}{r}}{q^{r}(q^{-1}t^{b+r};t)_{r}(q^{-1}t^{b+2r+1};t)_{a-r}}\omega\tilde{H}_{2^{a-r} 1^{b+2r}}[X;q^{-1},t^{-1}].
\end{align*}

Let us define the sum
\begin{align}\label{eq-Sa}
S_{a}:=\sum_{r=0}^{a}\frac{(-1)^{r}t^{\binom{r}{2}-\binom{a}{2}+\binom{a+b+r}{2}+\binom{a-r}{2}}\tbinom{a}{r}}{q^{r}(q^{-1}t^{b+r};t)_{r}(q^{-1}t^{b+2r+1};t)_{a-r}}\omega\tilde{H}_{2^{a-r} 1^{b+2r}}[X;q^{-1},t^{-1}],
\end{align}
where $S_{a,r}$ denotes its $r$-th summand. We claim that 
\begin{align}
S_{a} = t^{\binom{a+b}{2}}\omega \tilde{H}_{2^a1^b}[X;0,t^{-1}].
\end{align}
Establishing this identity directly yields the first equality in Theorem \ref{thm-nablaHt}, while the second equality follows from the combinatorial definition of the modified Hall--Littlewood polynomials. The validity of this claim relies on the two propositions presented below. Assuming these auxiliary results for the moment, the proof of the theorem is complete.
\end{proof}

Recall that, by Liouville's theorem in complex analysis, any rational function holomorphic on the entire Riemann sphere $\mathbb{C}\cup\{\infty\}$ must be constant. Therefore, to prove that $S_a$ is independent of $q$, it suffices to establish two facts: first, that $S_a$ has a finite limit as $q\to\infty$, and second, that $S_a$ has no finite poles (i.e., all potential finite poles are removable singularities). By a slight abuse of notation, we treat the expressions $S_a$ and $S_{a,r}$ directly as functions of a complex variable $q$. We formalize these two steps in the subsequent two propositions. Furthermore, the limit computed in the first proposition explicitly evaluates this constant.

\begin{proposition}
Let $S_{a}$ be the sum defined in \eqref{eq-Sa}. Then, we have
\begin{align*}
\lim_{q\to\infty}S_{a}=t^{\binom{a+b}{2}}\omega \tilde{H}_{2^a1^b}[X;0,t^{-1}].
\end{align*}
\end{proposition}
\begin{proof}
Observe that as $q\to\infty$ (i.e., $q^{-1}\to 0$), all terms in the summation corresponding to $r\geq 1$ vanish due to the $q^r$ factor in the denominator. Consequently, only the $r=0$ term survives, yielding
\begin{align*}
\lim_{q\to\infty}S_{a}=\lim_{q\to \infty}\frac{t^{-\binom{a}{2}+\binom{a+b}{2}+\binom{a}{2}}}{(q^{-1}t^{b+1};t)_a}\omega\tilde{H}_{2^a1^b}[X;q^{-1},t^{-1}]=t^{\binom{a+b}{2}}\omega \tilde{H}_{2^a1^b}[X;0,t^{-1}].
\end{align*}
This completes the proof.
\end{proof}

\begin{proposition}
The function $S_a$ is holomorphic on the complex plane $\mathbb{C}$.
\end{proposition}
\begin{proof}
As discussed, we must analyze the potential poles of $S_a$. These poles originate from the roots of the denominators in the summand $S_{a,r}$, namely $(q^{-1}t^{b+r};t)_{r}$ and $(q^{-1}t^{b+2r+1};t)_{a-r}$. The first denominator $(q^{-1}t^{b+r};t)_{r}$ introduces potential poles at $q=t^{p}$ for integers $p$ in the range $b+r\leq p\leq b+2r-1$. Similarly, the second denominator $(q^{-1}t^{b+2r+1};t)_{a-r}$ produces poles at $q=t^p$ for $p$ in the range $b+2r+1\leq p\leq a+b+r$. Our strategy is to demonstrate that these poles cancel pairwise between different summation terms in $S_a$.

Fix an integer $p$ such that $b\leq p\leq 2a+b$. For a given pole $q=t^p$, we categorize the relevant summation indices $r\in\{0,1,\dots,a\}$ into two disjoint sets, $A$ and $B$, corresponding to poles originating from the second and first denominators, 
\begin{align*}
A:&=\{r: 2r+1\leq p-b\leq a+r\},\\
B:&=\{r: r\leq p-b\leq 2r-1\}.
\end{align*}

Without loss of generality, assume that a term $S_{a,r}$ possesses a pole at $q=t^p$ originating from the second factor, meaning $r\in A$. This membership implies two crucial bounds: $2r<p-b$ and $r\geq p-b-a$. We now introduce a dual index $r':=p-b-r$. First, we must ensure that $r'$ is a valid summation index. Since $r\geq p-b-a$, we immediately obtain the upper bound
\begin{align*}
r'=p-b-r\leq p-b-(p-b-a)=a.    
\end{align*}
Coupled with $2r<p-b\implies r'>r\geq 0$, we confirm that $0 \leq r' \leq a$. Next, we verify that $r' \in B$. By definition, we need to check if $r' \leq p-b \leq 2r'-1$. The left inequality $r'\leq p-b$ simplifies to $p-b-r \leq p-b$, yielding $r \geq 0$, which is trivially true. For the right inequality, we substitute $r = p-b-r'$ into the original condition $2r+1\leq p-b$ to deduce
\begin{align*}
2(p-b-r')+1\leq p-b\implies p-b+1\leq 2r'.
\end{align*}
This implies $p-b \leq 2r'-1$. Thus, we have rigorously shown that $r\in A\implies r'\in B$. By the symmetry of the linear transformation $r' = p-b-r$, it is straightforward to verify the converse, establishing a bijection $r \in A \iff r' \in B$. Furthermore, since $2r<p-b$, we have $r'=p-b-r>r$, ensuring these paired indices are distinct. Consequently, whenever $S_{a,r}$ has a pole at $q=t^p$ from the second factor, there is a distinct, paired term $S_{a,r'}$ sharing the exact same pole originating from the first factor. This structural symmetry sets the stage for proving that their corresponding residues perfectly cancel.

If the denominator of $S_{a,r}$ does not vanish at $q=t^p$, then $\Res_{q=t^p}S_{a,r}=0$. Since nonzero residues only occur when $r\in A \cup B$, we obtain
\begin{align*}
\Res_{q=t^p}S_{a}=\sum_{r=0}^{a}\Res_{q=t^p}S_{a,r}=\sum_{r\in A\cup B}\Res_{q=t^p}S_{a,r}=\sum_{r\in A}(\Res_{q=t^p}S_{a,r}+\Res_{q=t^p}S_{a,r'})
\end{align*}
This exact cancellation, namely that $\Res_{q=t^p}S_{a,r}+\Res_{q=t^p}S_{a,r'}=0$ for each $r\in A$, is established in Proposition \ref{prop-Rescalculation}.
\end{proof}

To demonstrate that the sum of these residues vanishes, a crucial step is to analyze the contribution of the Macdonald polynomials within the summands.

\begin{lemma}\label{lem-wHta-ra-r'}
Fix $p$. For any index $r\in A$, let $r':=p-b-r$ be its paired index in $B$. Then we have
\begin{align*}
\omega\tilde{H}_{2^{a-r}1^{b+2r}}[X;t^{-p},t^{-1}]=\omega\tilde{H}_{2^{a-r'}1^{b+2r'}}[X;t^{-p},t^{-1}].
\end{align*}
\end{lemma}
\begin{proof}
By Lemma \ref{lem-Hlambda=Hmu}, it suffices to verify that
\begin{align*}
B_{2^{a-r}1^{b+2r}}(t^{-p},t^{-1})=B_{2^{a-r'}1^{b+2r'}}(t^{-p},t^{-1}).
\end{align*}

For a general two-column partition $\mu=(2^k1^\ell)$, evaluating $B_{\mu}(q,t)=\sum_{c\in\mu}q^{\mathrm{coarm}_{\mu}(c)}t^{\mathrm{coleg}_{\mu}(c)}$ at $(q,t)=(t^{-p},t^{-1})$ yields a sum of two geometric progressions:
\begin{equation*}
B_{2^k1^\ell}(t^{-p},t^{-1})=\sum_{i=0}^{k+\ell-1}t^{-i}+\sum_{i=0}^{k-1}t^{-(p+i)}.
\end{equation*}

Substituting $k=a-r$ and $\ell=b+2r$, and shifting the summation index in the second term via $j=p+i$, we obtain
\begin{align*}
B_{2^{a-r}1^{b+2r}}(t^{-p},t^{-1})
&=\sum_{i=0}^{a+b+r-1}t^{-i}+\sum_{i=0}^{a-r-1}t^{-(p+i)}=\sum_{i=0}^{a+b+r-1}t^{-i}+\sum_{j=p}^{p+a-r-1}t^{-j}\\
&= \sum_{i=0}^{a+b+r-1}t^{-i}+\sum_{j=p}^{a+b+r'-1}t^{-j},   
\end{align*}
where the last equality follows from the relation $p+a-r-1=a+b+r'-1$.
Similarly, we have
\begin{align*}\label{eq-B-rho-sum}
B_{2^{a-r'}1^{b+2r'}}(t^{-p},t^{-1})=\sum_{i=0}^{a+b+r'-1}t^{-i}+\sum_{j=p}^{a+b+r-1}t^{-j}.
\end{align*}
Recall that $r'>r$ for any $r\in A$. Taking the difference between the two expressions, the common terms cancel out, leaving
\begin{align*}
B_{2^{a-r'}1^{b+2r'}}(t^{-p},t^{-1})-B_{2^{a-r}1^{b+2r}}(t^{-p},t^{-1})=\sum_{k=a+b+r}^{a+b+r'-1}t^{-k}-\sum_{k=a+b+r}^{a+b+r'-1}t^{-k}=0.
\end{align*}
This completes the proof.
\end{proof}

\begin{proposition}\label{prop-Rescalculation}
Following the notation established above, we have
\begin{align*}
Res_{q=t^p}S_{a,r}+Res_{q=t^p}S_{a,r'}=0.
\end{align*}
\end{proposition}
\begin{proof}
We begin by evaluating the residue of $S_{a,r}$ at $q=t^p$. Applying \eqref{eq-punctured} with $x=q^{-1}t^{b+2r+1}$, $k=p-b-2r-1$, and $n=a-r$, we isolate the vanishing factor $(1-q^{-1}t^p)$ in the denominator:
\begin{align*}
(q^{-1}t^{b+2r+1};t)_{a-r}&=(1-q^{-1}t^p)(q^{-1}t^{b+2r+1};t)_{p-b-2r-1}(q^{-1}t^{p+1};t)_{a-p+b+r}.
\end{align*}
Evaluating the limit as $q\to t^p$ then yields the residue:
\begin{align*}
\Res_{q=t^p}S_{a,r}&=\lim_{q\to t^p}(1-q^{-1}t^p) \frac{(-1)^{r}t^{\binom{r}{2}-\binom{a}{2}+\binom{a+b+r}{2}+\binom{a-r}{2}}\tbinom{a}{r}}{q^{r}(q^{-1}t^{b+r};t)_{r}(q^{-1}t^{b+2r+1};t)_{a-r}}\omega\tilde{H}_{2^{a-r} 1^{b+2r}}[X;q^{-1},t^{-1}]\\
&= \frac{(-1)^{r}t^{\binom{r}{2}-\binom{a}{2}+\binom{a+b+r}{2}+\binom{a-r}{2}}\tbinom{a}{r}}{t^{pr}(t^{b+r-p};t)_{r}(t^{b+2r+1-p};t)_{p-b-2r-1}(t;t)_{a-p+b+r}}\omega\tilde{H}_{2^{a-r}1^{b+2r}}[X;t^{-p},t^{-1}].
\end{align*}
Similarly, for the paired index $r'=p-b-r\in B$, the vanishing factor resides in the first shifted factorial:
\begin{align*}
(q^{-1}t^{b+r'};t)_{r'} &= (1-q^{-1}t^p)(q^{-1}t^{p-r};t)_{r}(q^{-1}t^{p+1};t)_{p-b-2r-1}.
\end{align*}
Evaluating the associated limit, we obtain
\begin{align*}
\Res_{q=t^p}S_{a,r'} &= \lim_{q\to t^p}(1-q^{-1}t^p) \frac{(-1)^{r'}t^{\binom{r'}{2}-\binom{a}{2}+\binom{a+b+r'}{2}+\binom{a-r'}{2}}\tbinom{a}{r'}}{q^{r'}(q^{-1}t^{b+r'};t)_{r'}(q^{-1}t^{b+2r'+1};t)_{a-r'}}\omega\tilde{H}_{2^{a-r'} 1^{b+2r'}}[X;q^{-1},t^{-1}]\\
&= \frac{(-1)^{r'}t^{\binom{r'}{2}-\binom{a}{2}+\binom{a+b+r'}{2}+\binom{a-r'}{2}}\tbinom{a}{r'}}{t^{pr'}(t^{-r};t)_{r}(t;t)_{p-b-2r-1}(t^{p-b-2r+1};t)_{a-p+b+r}}\omega\tilde{H}_{2^{a-r'} 1^{b+2r'}}[X;t^{-p},t^{-1}].
\end{align*}

By Lemma \ref{lem-wHta-ra-r'}, the Macdonald polynomials in both residues evaluate to the identical expression. It therefore suffices to demonstrate that their remaining rational coefficients sum to zero. To this end, we apply \eqref{eq-inversion} and \eqref{eq-quotient-neg} to convert all $t$-shifted factorials with negative powers in the denominators into standard factorials $(t;t)_k$. For $\Res_{q=t^p}S_{a,r}$, we have
\begin{align*}
(t^{b+r-p};t)_{r}&=(-1)^{r}t^{r(b+r-p)+\binom{r}{2}}\frac{(t;t)_{p-b-r}}{(t;t)_{p-b-2r}},\\
(t^{b+2r+1-p};t)_{p-b-2r-1}&=(-1)^{p-b-2r-1}t^{-\binom{p-b-2r}{2}}(t;t)_{p-b-2r-1}.
\end{align*}
Likewise, for $\Res_{q=t^p}S_{a,r'}$, we obtain
\begin{align*}
(t^{-r};t)_{r}&=(-1)^{r}t^{-\binom{r+1}{2}}(t;t)_{r},\\
(t^{p-b-2r+1};t)_{a-p+b+r}&=\frac{(t;t)_{a-r}}{(t;t)_{p-b-2r}}.
\end{align*}
Expanding the $t$-binomial coefficients $\tbinom{a}{r}$ and $\tbinom{a}{r'}$ reveals a striking structural symmetry: both residue expressions share the exact same rational factor
\begin{align*}
\frac{(t;t)_a(t;t)_{p-b-2r}}{(t;t)_r(t;t)_{a-r}(t;t)_{p-b-r}(t;t)_{p-b-2r-1}(t;t)_{a-p+b+r}}.
\end{align*}
Factoring this out, the verification reduces to comparing the overall signs and the remaining powers of $t$. A direct computation confirms that the exponents of $t$ match identically, while their signs are strictly opposite. This exact cancellation establishes that $\Res_{q=t^p}S_{a,r}+\Res_{q=t^p}S_{a,r'}=0$, thereby completing the proof.
\end{proof}

\section{An Alternative Proof of the First Part of Conjecture \ref{cnj-3} in the Two-Column Case}

\begin{theorem}
Let $a$ and $b$ be nonnegative integers. For any integer $j$ such that $0\leq j\leq a$,
\begin{align*}
\nabla\omega\tilde{H}_{2^{a-j}1^{b+2j}}[X;0,t^{-1}]=\sum_{i=j}^{a}M_{i,j}\omega\tilde{H}_{2^{a-i}1^{b+2i}}[X;0,t^{-1}].
\end{align*}
\end{theorem}

\begin{proof}
By Theorem \ref{thm-nablaHt} and \eqref{eq-HL2Mac}, the left-hand side (LHS) can be expanded as
\begin{align*}
\mathrm{LHS}&=\frac{(-q)^{j-a}}{t^{\binom{2a+b}{2}+\binom{a+b+j}{2}}}\sum_{m=0}^{a-j}\frac{(-q)^{-m}t^{\binom{a+b+j+1}{2}-\binom{a+b+j+1-m}{2}}\tbinom{a-j}{m}}{(q^{-1}t^{2a+b-2m+1};t)_m(q^{-1}t^{a+b+j-m};t)_{a-j-m}}\nabla^2\tilde{H}_{2^m1^{2a+b-2m}}[X;q,t]\\
&=\frac{(-q)^{j-a}}{t^{\binom{2a+b}{2}+\binom{a+b+j}{2}}}\sum_{r=j}^{a}\frac{(-q)^{r-a}t^{\binom{a+b+j+1}{2}-\binom{b+j+r+1}{2}}\tbinom{a-j}{a-r}}{(q^{-1}t^{b+2r+1};t)_{a-r}(q^{-1}t^{b+j+r};t)_{r-j}}T_{2^{a-r}1^{b+2r}}^{2}\tilde{H}_{2^{a-r}1^{b+2r}}[X;q,t]
\end{align*}
On the other hand, the right-hand side (RHS) is given by
\begin{align*}
\mathrm{RHS}&=\sum_{i=j}^{a}\frac{M_{i,j}(-q)^{i-a}}{t^{\binom{2a+b}{2}+\binom{a+b+i}{2}}}\sum_{r=i}^{a}\frac{(-q)^{r-a}t^{\binom{a+b+i+1}{2}-\binom{b+i+r+1}{2}}\tbinom{a-i}{a-r}}{(q^{-1}t^{b+2r+1};t)_{a-r}(q^{-1}t^{b+i+r};t)_{r-i}}T_{2^{a-r}1^{b+2r}}\tilde{H}_{2^{a-r}1^{b+2r}}[X;q,t]\\
&=\sum_{r=j}^{a}\sum_{i=j}^{r}\frac{M_{i,j}(-q)^{i-a}}{t^{\binom{2a+b}{2}+\binom{a+b+i}{2}}}\frac{(-q)^{r-a}t^{\binom{a+b+i+1}{2}-\binom{b+i+r+1}{2}}\tbinom{a-i}{a-r}}{(q^{-1}t^{b+2r+1};t)_{a-r}(q^{-1}t^{b+i+r};t)_{r-i}}T_{2^{a-r}1^{b+2r}}\tilde{H}_{2^{a-r}1^{b+2r}}[X;q,t]
\end{align*}
By equating the coefficients of $\tilde{H}_{2^{a-r}1^{b+2r}}[X;q,t]$ on both sides for each $j\leq r\leq a$ and canceling a single common factor of $T_{2^{a-r}1^{b+2r}}$, the problem reduces to verifying the following identity:
\begin{align*}
\frac{(-q)^{j-a}}{t^{\binom{2a+b}{2}+\binom{a+b+j}{2}}}&\frac{(-q)^{r-a}t^{\binom{a+b+j+1}{2}-\binom{b+j+r+1}{2}}\tbinom{a-j}{a-r}}{(q^{-1}t^{b+2r+1};t)_{a-r}(q^{-1}t^{b+j+r};t)_{r-j}}T_{2^{a-r}1^{b+2r}}\\
&=\sum_{i=j}^{r}\frac{M_{i,j}(-q)^{i-a}}{t^{\binom{2a+b}{2}+\binom{a+b+i}{2}}}\frac{(-q)^{r-a}t^{\binom{a+b+i+1}{2}-\binom{b+i+r+1}{2}}\tbinom{a-i}{a-r}}{(q^{-1}t^{b+2r+1};t)_{a-r}(q^{-1}t^{b+i+r};t)_{r-i}}.
\end{align*}

Let $\mathrm{RHS}_{\mathrm{coeff}}$ denote the right-hand side of this identity. Substituting the explicit formula 
\begin{align*}
M_{i,j}=q^{a-i}t^{\binom{a-i}{2}+\binom{a+b+i}{2}}\tbinom{a-j}{i-j},   
\end{align*}
and applying the $t$-binomial relation:
\begin{align*}
\tbinom{a-j}{i-j}\tbinom{a-i}{a-r} = \tbinom{a-j}{r-j}\tbinom{r-j}{i-j},
\end{align*}
from Lemma \ref{lem-tbinomexchange}, the expression compacts to:
\begin{align*}
\mathrm{RHS}_{\mathrm{coeff}}&=\frac{(-q)^{r-a}\tbinom{a-j}{r-j}}{t^{\binom{2a+b}{2}}(q^{-1}t^{b+2r+1};t)_{a-r}} \sum_{i=j}^{r}\frac{(-1)^{a-i}t^{\binom{a-i}{2}+\binom{a+b+i+1}{2}-\binom{b+i+r+1}{2}}}{(q^{-1}t^{b+i+r};t)_{r-i}}\tbinom{r-j}{i-j}\\
&=\frac{(-q)^{r-a}\tbinom{a-j}{r-j}}{t^{\binom{2a+b}{2}}(q^{-1}t^{b+2r+1};t)_{a-r}} \sum_{s=0}^{r-j}\frac{(-1)^{a-s-j}t^{\binom{a-j-s}{2}+\binom{a+b+j+s+1}{2}-\binom{b+j+s+r+1}{2}}}{(q^{-1}t^{b+j+s+r};t)_{r-j-s}}\tbinom{r-j}{s}.
\end{align*}
By Lemma \ref{lem-A+-B}, the exponent of $t$ separates neatly into a constant part and an $s$-dependent part: 
\begin{align*}
\binom{a-j-s}{2}&+\binom{a+b+j+s+1}{2}-\binom{b+j+s+r+1}{2}\\
&=\binom{a-j}{2}+\binom{a+b+j+1}{2}-\binom{b+j+r+1}{2}+\binom{s+1}{2}-(r-j)s.
\end{align*}
Pulling the constant powers of $t$ and alternating signs out of the summation yields:
\begin{align*}
\mathrm{RHS}_{\mathrm{coeff}} &=\frac{(-1)^{a-j}(-q)^{r-a}t^{\binom{a-j}{2}+\binom{a+b+j+1}{2}}\tbinom{a-j}{r-j}}{t^{\binom{2a+b}{2}+\binom{b+j+r+1}{2}}(q^{-1}t^{b+2r+1};t)_{a-r}}\sum_{s=0}^{r-j}\frac{(-1)^{s}t^{\binom{s+1}{2}-(r-j)s}}{(q^{-1}t^{b+j+s+r};t)_{r-j-s}}\tbinom{r-j}{s}\\
&=\frac{(-1)^{a-j}(-q)^{r-a}t^{\binom{a-j}{2}+\binom{a+b+j+1}{2}}\tbinom{a-j}{r-j}}{t^{\binom{2a+b}{2}+\binom{b+j+r+1}{2}}(q^{-1}t^{b+2r+1};t)_{a-r}(q^{-1}t^{b+j+r};t)_{r-j}}\\
&\quad \times\sum_{s=0}^{r-j}(-1)^{s}t^{\binom{s+1}{2}-(r-j)s}(q^{-1}t^{b+j+r};t)_{s}\tbinom{r-j}{s}\\
&=\frac{(-1)^{a-j}(-q)^{r-a}t^{\binom{a-j}{2}+\binom{a+b+j+1}{2}}\tbinom{a-j}{r-j}(q^{-1}t^{b+j+r})^{r-j}}{t^{\binom{2a+b}{2}+\binom{b+j+r+1}{2}}(q^{-1}t^{b+2r+1};t)_{a-r}(q^{-1}t^{b+j+r};t)_{r-j}}
\end{align*}
The second step is justified by \eqref{eq-splitting}, and the final step follows directly from Lemma \ref{lem-inversetbinomial}. 

Substituting this evaluation back and applying the algebraic identity
\begin{align*}
\binom{a-r}{2}+\binom{a+b+r}{2}-\binom{a+b+j}{2}=\binom{a-j}{2}+(b+j+r)(r-j),   
\end{align*}
to simplify the exponents of $t$, we perfectly recover the initial expression for the left-hand side, accounting for the common factor $T_{2^{a-r}1^{b+2r}}$. Since the coefficients coincide, the theorem holds.
\end{proof}

Having established the first half of Conjecture \ref{cnj-3}, we defer the formulation of the second half to the next section, where it will be presented in a more general setting.

\section{The action of \texorpdfstring{$\nabla^k$}{nablak} on two-column modified Hall--Littlewood polynomials}

In general, the Schur positivity of $\nabla F[X]$ does not imply that higher powers $\nabla^k F[X]$ will remain Schur positive. As a counterexample, consider the symmetric function
\begin{align*}
F[X]=(q^2t+qt^2+q^2+qt+t^2)s_2+(q^2t^2+q^2t+qt^2)s_{1,1}.
\end{align*}
Direct computation reveals that while the first application of the nabla operator yields a Schur positive function,
\begin{align*}
\nabla F[X]=(q^2t^2+q^2t+qt^2)s_2+ q^2t^2s_{1,1},
\end{align*}
a second application gives $\nabla^2 F[X]=q^2t^2s_{2}-q^3t^3s_{1,1}$, which is clearly not Schur positive due to the negative coefficient.

In this section, we demonstrate that Conjectures \ref{cnj-2} and \ref{cnj-3} remain valid in the two-column case when $\nabla$ is replaced by $\nabla^k$.

Recall that $M$ is an $(a+1)\times (a+1)$ lower triangular matrix whose rows and columns are indexed by the set $\{0,1,\cdots,a\}$, with its $(i,j)$-th entry given by $M_{i,j}=T_{2^{a-i}1^{b+2i}}\tbinom{a-j}{i-j}$ for $i\geq j$, and $0$ otherwise. We denote the $(i,j)$-th entry of its $k$-th power $M^k$ by $M_{i,j}^{(k)}$.

\begin{corollary}\label{cor-nablakomegaHt}
For any nonnegative integers $a$ and $b$, and any positive integer $k$, then
\begin{align*}
\nabla^k\omega \tilde{H}_{2^a1^b}[X;0,t^{-1}]&=\sum_{i=0}^{a}M_{i,0}^{(k)}\omega\tilde{H}_{2^{a-i}1^{b+2i}}[X;0,t^{-1}]\\
&=\sum_{i=0}^{a}\left(\sum_{0=s_0\leq s_1\leq\cdots\leq s_{k}=i}\prod_{\rho=1}^{k}M_{s_{\rho},s_{\rho-1}}\right)\omega\tilde{H}_{2^{a-i}1^{b+2i}}[X;0,t^{-1}].
\end{align*}
In particular, $\nabla^k\omega \tilde{H}_{2^a1^b}[X;0,t^{-1}]$ is Schur positive.
\end{corollary}
\begin{proof}
We proceed by induction on $k$. For the base case $k=1$, the statement reduces to the established rule for a single application of the $\nabla$ operator
\begin{align*}
\nabla\omega\tilde{H}_{2^a1^b}[X;0,t^{-1}]=\sum_{i=0}^{a}M_{i,0}\omega\tilde{H}_{2^{a-i}1^{b+2i}}[X;0,t^{-1}],    
\end{align*}
which holds trivially since $M_{i,0}^{(1)}=M_{i,0}=T_{2^{a-i}1^{b+2i}}\tbinom{a}{i}\in\mathbb{N}[q,t]$. Assume that the identity holds for some integer $k-1\geq 1$, that is
\begin{align*}
\nabla^{k-1}\omega \tilde{H}_{2^a1^b}[X;0,t^{-1}]=\sum_{j=0}^{a}M_{j,0}^{(k-1)}\omega\tilde{H}_{2^{a-j}1^{b+2j}}[X;0,t^{-1}].    
\end{align*}
Applying the $\nabla$ operator to this inductive hypothesis and utilizing its linearity, we obtain
\begin{align*}
\nabla^k\omega \tilde{H}_{2^a1^b}[X;0,t^{-1}]&=\sum_{j=0}^{a}M_{j,0}^{(k-1)}\nabla\big(\omega\tilde{H}_{2^{a-j}1^{b+2j}}[X;0,t^{-1}]\big)\\
&=\sum_{j=0}^{a}M_{j,0}^{(k-1)}\sum_{i=j}^{a}M_{i,j}\omega\tilde{H}_{2^{a-i}1^{b+2i}}[X;0,t^{-1}]\\
&=\sum_{i=0}^{a}\left(\sum_{j=0}^{i}M_{i,j}M_{j,0}^{(k-1)}\right)\omega\tilde{H}_{2^{a-i}1^{b+2i}}[X;0,t^{-1}].
\end{align*}
By the definition of matrix multiplication, the inner sum $\sum_{j=0}^{i} M_{i,j} M_{j,0}^{(k-1)}$ is precisely $M_{i,0}^{(k)}$. This immediately establishes the first equality. Expanding this matrix entry as a sum over all valid intermediate indices $s_{\rho}$ yields the expanded product form. Furthermore, the Schur positivity of the result follows directly from that of the modified Hall--Littlewood polynomials.
\end{proof}

\begin{corollary}
For any nonnegative integers $a$ and $b$, let $\mu=(2^{a}1^{b})$ and $\nu$ be two-column partitions of the same integer. If $\mu\unrhd\nu$, then for any positive integer $k$, the difference
\begin{align*}
\nabla^{k}\omega\tilde{H}_{\mu}[X;0,t^{-1}]-\nabla^{k}\omega\tilde{H}_{\nu}[X;0,t^{-1}]
\end{align*}
is Schur positive.
\end{corollary}

\begin{proof}
By the transitivity of the dominance order, it suffices to prove the statement for adjacent two-column partitions, namely $\mu=(2^{a-j}1^{b+2j})$ and $\nu=(2^{a-j-1}1^{b+2j+2})$ for any $0\leq j<a$. As established in the preceding discussion, the action of $\nabla^k$ on these basis elements is governed by the transition matrix $M^k$. Consequently, the difference can be expressed as
\begin{align*}
\nabla^{k}\omega\tilde{H}_{\mu}[X;0,t^{-1}]-\nabla^{k}\omega\tilde{H}_{\nu}[X;0,t^{-1}]= \sum_{i=0}^{a}\left( M_{i,j}^{(k)}-M_{i,j+1}^{(k)}\right)\omega\tilde{H}_{2^{a-i}1^{b+2i}}[X;0,t^{-1}].
\end{align*}
Since the basis elements $\omega\tilde{H}_{\lambda}[X;0,t^{-1}]$ are Schur positive, the result follows if we can show that the coefficient difference $M_{i,j}^{(k)} - M_{i,j+1}^{(k)}$ belongs to $\mathbb{N}[q,t]$. We proceed by induction on $k$.

For the base case $k=1$, the matrix entries $M_{i,j}$ are given explicitly. Utilizing a standard $t$-binomial identity, the difference simplifies as
\begin{align*}
M_{i,j}-M_{i,j+1}&=T_{2^{a-i}1^{b+2i}}\left(\tbinom{a-j}{i-j}-\tbinom{a-j-1}{i-j-1}\right)\\
&=T_{2^{a-i}1^{b+2i}}t^{i-j}\tbinom{a-j-1}{i-j},
\end{align*}
which clearly belongs to $\mathbb{N}[q,t]$.

For the inductive step, assume that $M_{r,j}^{(k-1)}-M_{r,j+1}^{(k-1)} \in \mathbb{N}[q,t]$ for all valid indices $r$. Using the matrix recurrence $M^k = M \cdot M^{k-1}$, we have:
\begin{align*}
M_{i,j}^{(k)}-M_{i,j+1}^{(k)}&=\sum_{r=j}^{i}M_{i,r}M_{r,j}^{(k-1)}-\sum_{r=j+1}^{i}M_{i,r}M_{r,j+1}^{(k-1)}\\
&=M_{i,j}M_{j,j}^{(k-1)}+\sum_{r=j+1}^{i}M_{i,r}\left(M_{r,j}^{(k-1)}-M_{r,j+1}^{(k-1)}\right).
\end{align*}
Since all entries of $M$ are polynomials in $\mathbb{N}[q,t]$, the entries of its power $M^{k-1}$, including $M_{j,j}^{(k-1)}$, also belong to $\mathbb{N}[q,t]$. Furthermore, the term $ M_{r,j}^{(k-1)}-M_{r,j+1}^{(k-1)}$ lies in $\mathbb{N}[q,t]$ by the inductive hypothesis. Because $\mathbb{N}[q,t]$ is closed under addition and multiplication, the entire expression is evaluated as a polynomial with nonnegative coefficients. This establishes that $M_{i,j}^{(k)}-M_{i,j+1}^{(k)}\in\mathbb{N}[q,t]$ for all $k\geq 1$, completing the proof.
\end{proof}

As a direct consequence of Theorem \ref{thm-nablaHt} and Corollary \ref{cor-nablakomegaHt}, we obtain the following result.

\begin{corollary}\label{cor-nablakab}
For any nonnegative integers $a$ and $b$, and any positive integer $k$,
\begin{align*}
(-1)^{a}\nabla^{k}\tilde{H}_{2^a1^b}[X;0,t]&=\sum_{i=0}^{a}c_i^{(k)}\omega\tilde{H}_{2^{a-i}1^{b+2i}}[X;0,t^{-1}]\\
&=\sum_{i=0}^{a}c_i^{(k)}\sum_{\lambda\vdash 2a+b}\sum_{T\in \mathrm{SSYT}(\lambda,2^{a-i}1^{b+2i})}t^{-\mathrm{cocharge}(T)}s_{\lambda'}
\end{align*}
where the coefficients $c_i^{(k)}\in\mathbb{N}[q,t]$ are given by the explicit formula
\begin{align*}
c_i^{(k)}&=q^{a}t^{\binom{2a+b}{2}+\binom{a+b}{2}}\sum_{0=s_0\leq s_1\leq\cdots\leq s_{k-1}=i}\prod_{\rho=1}^{k-1}M_{s_{\rho},s_{\rho-1}}.
\end{align*}
In particular, $(-1)^{a}\nabla^{k}\tilde{H}_{2^a1^b}[X;0,t]$ is Schur positive.
\end{corollary}

Specializing to $k=1$, the empty product convention dictates that $M_{i,0}^{(0)}=\delta_{i,0}$. Consequently, we see that $c_{i}^{(1)}=0$ for $i>0$, while $c_{0}^{(1)}=q^{a}t^{\binom{2a+b}{2}+\binom{a+b}{2}}$. Thus, Corollary \ref{cor-nablakab} reduces to Theorem \ref{thm-nablaHt}.

\bigskip
\noindent
\textbf{Acknowledgements:} We would like to thank Michele D'Adderio for valuable conversations.

\appendix
\section{Alternative proof of Theorem \ref{thm-nablaHt}}
This appendix presents a concise, self-contained proof of Theorem \ref{thm-nablaHt}. To avoid any conflict with our earlier notation, we adopt a slight modification of the conventions used in \cite{MR3787405}. Specifically, we define the linear operator $\nabla': \Lambda\to\Lambda$ by 
\begin{align*}
\nabla' \tilde{H}_{\mu}[X;q,t]:=(-1)^{|\mu|}T_{\mu}\tilde{H}_{\mu}[X;q,t].
\end{align*}
For $F[X;q,t]\in \Lambda^{(n)}$, define 
\begin{align*}
\overline{\omega}F[X;q,t]:=F[-X;q^{-1},t^{-1}]=(-1)^{n}\omega F[X;q^{-1},t^{-1}].
\end{align*}

To state some of our results, we need to use the language of plethystic substitution. Let $E(t_{1},t_{2},t_{3},\dots)$ be a formal series of rational functions in the parameters $t_{1},t_{2},t_{3},\cdots$. Define the \emph{plethystic substitution} of $E$ into $p_{k}$, denoted by $p_{k}[E]$, as
\begin{align*}
p_{k}[E]=E(t_{1}^{k},t_{2}^{k},\dots).
\end{align*}
We refer the reader to \cite{MR2371044}*{Section 1} for a detailed overview of the usual plethystic substitution in symmetric functions. We recall the definitions of the \emph{creation operators} $\mathbb{H}_m$ and $\mathbb{B}_m$ for $m\in\mathbb{Z}$:
\begin{align*}
\mathbb{H}_{m}F[X] &:= F[X-(1-q)z^{-1}]\mathrm{Exp}[zX]|_{z^{m}},\\
\mathbb{B}_{m}F[X] &:= F[X-(q-1)z^{-1}]\mathrm{Exp}[-zX]|_{z^{m}}.
\end{align*}
where $\mathrm{Exp}[X]:=\sum_{i\geq 0}h_{i}[X]$ is the plethystic exponential. See \cite{MR2957232} for related variations. 

Note that these operators satisfy the relation
\begin{align*}
\mathbb{B}_m=(-1)^{m}\omega\mathbb{H}_m\omega.    
\end{align*}
Furthermore, it has been shown that the operators $\mathbb{H}_m$ act as creation operators for a family of Hall--Littlewood polynomials, typically denoted by $Q_{\mu'}[X;q]:=q^{n(\mu)}\tilde{H}_{\mu}[X;0,q^{-1}]$.

\begin{lemma}\cite{MR1112626}\label{lem-HHtoHt}
Let $\mu=(\mu_1,\mu_2,\dots,\mu_{\ell})$ be a partition. Then
\begin{align}
\mathbb{H}_{\mu_1}\mathbb{H}_{\mu_2}\cdots\mathbb{H}_{\mu_{\ell}}(1)=q^{n(\mu)}\tilde{H}_{\mu}[X;0,q^{-1}].
\end{align}
\end{lemma}

We now briefly recall the formalism of the Dyck path algebra. 

For a polynomial $P$ in the variables $u$ and $v$, define the operator $\Upsilon_{uv}$ as
\begin{align*}
(\Upsilon_{uv}P)(u,v):=\frac{(q-1)uP(u,v)+(v-qu)P(v,u)}{v-u}.
\end{align*}
Consider the space $V_k:=\Lambda[y_1,y_2,\cdots,y_k]=\Lambda\otimes \mathbb{Q}[y_1,\cdots,y_k]$ and let $V_{*}:=\oplus_{k\geq 0} V_k$. On these spaces, we define the operators
\begin{align*}
T_i:=\Upsilon_{y_{i}y_{i+1}}: V_{k}\to V_{k} \quad \text{for } 1\leq i\leq k-1.
\end{align*}
Furthermore, we introduce the \emph{raising and lowering operators} $d_{+}: V_k\to V_{k+1}$ and $d_{-}: V_k\to V_{k-1}$. For any $F[X]\in V_k$, their actions are explicitly given by
\begin{align*}
(d_{+}F)[X]:&=T_{1}T_{2}\cdots T_{k}(F[X+(q-1)y_{k+1}]),\\
(d_{-}F)[X]:&=-F[X-(q-1)y_k]\sum_{i\geq 0}(-1/y_k)^{i}e_{i}[X]|_{y_{k}^{-1}}.
\end{align*}

The \emph{Dyck path algebra} $\mathbb{A}=\mathbb{A}_q$ is the path algebra of the quiver with vertex set $\mathbb{Z}_{\geq 0}$, arrows $d_{+}$ from $i$ to $i+1$, arrows $d_{-}$ from $i+1$ to $i$ for $i\in\mathbb{Z}_{\geq 0}$, and loops $T_1,T_2,\cdots,T_{k-1}$ from $k$ to $k$ subject to some relations which is listed in \cite{MR3787405}*{Definition 5.1}. Let $\mathbb{A}^{*}=\mathbb{A}_{q^{-1}}$ and label the corresponding generators by $d_{+}^{*}$, $d_{-}^{*}$, ${T_{i}^{*}}$ and $e_{i}^{*}$. Denote by $z_i$ the image of $y_i$ under the isomorphism from $\mathbb{A}$ to $\mathbb{A}^{*}$ which sends generators to generators, and which is antilinear with respect to $q\mapsto q^{-1}$.

To facilitate our subsequent proofs, we record several key commutation relations and properties of the automorphism $\mathcal{N}$.

\begin{lemma}\cite{MR3787405}*{Theorem 6.1}\label{lem-d+yi}
When acting on $V_{k}$, the following commutation relations hold:
\begin{align}
z_{1}d_{+}=-y_{1}d_{+}^{*}tq^{k+1},\quad z_{i+1}d_{+}=d_{+}z_{i}, \quad  d_{+}^{*}y_i=y_{i+1}d_{+}^{*}.
\end{align}
\end{lemma}

\begin{lemma}\cite{MR3787405}*{Theorem 7.4}\label{lem-mathcalN}
There exists a unique antilinear degree-preserving automorphism from $V_{*}$ to $V_{*}$ satisfying
\begin{align}
\mathcal{N}(1)=1, \quad \mathcal{N}T_{i}=T_{i}^{-1},\quad \mathcal{N}d_{-}=d_{-}\mathcal{N},\quad \mathcal{N}d_{+}=d_{+}^{*}\mathcal{N},\quad \mathcal{N}y_{i}=z_{i}\mathcal{N}.
\end{align}
Moreover, $\mathcal{N}$ is an involution (that is $\mathcal{N}^2=id$), and it restricts to $\mathcal{N}=\nabla'\overline{\omega}$ on $V_{0}=\Lambda$.
\end{lemma}

To apply these relations, we first express the two-column modified Hall--Littlewood polynomials in terms of raising and lowering operators.

\begin{lemma}\label{lem-dplusdminustoHL}
We have
\begin{align*}
d_{-}^{a+b}y_1y_2\cdots y_a (d_{+}^{*})^{a+b}(1)=(-1)^{a}q^{n(2^a1^b)}\omega\tilde{H}_{2^a1^b}[X;0,q^{-1}].
\end{align*}
\end{lemma}
\begin{proof}
In the proof of Lemma 5.6 in \cite{MR3787405}, it was shown that
\begin{align*}
d_{-}^{m}y_{1}^{r_1}y_{2}^{r_2}\cdots y_{k+m}^{r_{k+m}}d_{+}^{k+m}(1)=(-1)^{m}y_{1}^{r_1}y_{2}^{r_2}\cdots y_{k}^{r_k}\mathbb{B}_{r_{k+1}+1}\mathbb{B}_{r_{k+2}+1}\cdots \mathbb{B}_{r_{k+m}+1}(1).
\end{align*}
Setting $k=0$, $m=a+b$, $r_1=\cdots=r_a=1$, and $r_{a+1}=\cdots=r_{a+b}=0$, we obtain
\begin{align*}
d_{-}^{a+b}y_1y_2\cdots y_a (d_{+}^{*})^{a+b}(1)&=d_{-}^{a+b}y_1y_2\cdots y_a d_{+}^{a+b}(1)\\
&=(-1)^{a+b}\mathbb{B}_{2}^{a}\mathbb{B}_{1}^{b}(1)\\
&=(-1)^{a+b}(-1)^{b}\omega\mathbb{H}_{2}^{a}\mathbb{H}_{1}^{b}\omega(1)\\
&=(-1)^{a}q^{n(2^a1^b)}\omega\tilde{H}_{2^a1^b}[X;0,q^{-1}].
\end{align*}
Here, the first equality holds because the operators $d_{+}$ and $d_{+}^{*}$ act trivially on $1$, meaning $d_{+}^{a+b}(1)=(d_{+}^{*})^{a+b}(1)=1$. The final equality follows from Lemma \ref{lem-HHtoHt}.
\end{proof}

\begin{proof}[Second proof of Theorem 4.1]
First, observe that $d_{-}^{a+b}y_1y_2\cdots y_a (d_{+}^{*})^{a+b}(1)\in \Lambda$. By Lemma \ref{lem-dplusdminustoHL} and \ref{lem-mathcalN}, we obtain
\begin{align*}
\mathcal{N}d_{-}^{a+b}y_1y_2\cdots y_a(d_{+}^{*})^{a+b}(1)&=\mathcal{N}((-1)^{a}q^{n(2^a1^b)}\omega\tilde{H}_{2^a1^b}[X;0,q^{-1}])\\
&=\nabla'((-1)^{a}q^{-n(2^a1^b)}\omega\tilde{H}_{2^a1^b}[-X;0,q])\\
&=\nabla'((-1)^{a}q^{-n(2^a1^b)}(-1)^{2a+b}\tilde{H}_{2^a1^b}[X;0,q])\\
&=(-1)^{a}q^{-n(2^a1^b)}\nabla\tilde{H}_{2^a1^b}[X;0,q].
\end{align*}

On the other hand, utilizing the commutation relations in Lemma \ref{lem-d+yi}, we have
\begin{align*}
d_{-}^{a+b}y_1y_2\cdots y_a(d_{+}^{*})^{a+b}(1)&=d_{-}^{a+b}y_1y_2\cdots y_{a-1}d_{+}^{*}y_{a-1}(d_{+}^{*})^{a+b-1}(1)\\
&=d_{-}^{a+b}y_{1}d_{+}^{*}y_1y_2\cdots y_{a-1}(d_{+}^{*})^{a+b-1}(1)\\
&=d_{-}^{a+b}(y_{1}d_{+}^{*})^{a}(d_{+}^{*})^{b}(1)\\
&=(-1)^{a}t^{-a}q^{\sum_{k=b}^{b+a-1}(-k-1)}d_{-}^{a+b}(z_{1}d_{+})^{a}(d_{+}^{*})^{b}(1)\\
&=(-1)^{a}t^{-a}q^{\binom{b+1}{2}-\binom{a+b+1}{2}}d_{-}^{a+b}z_1z_2\cdots z_{a}d_{+}^{a}(d_{+}^{*})^{b}(1)\\
&=(-1)^{a}t^{-a}q^{\binom{b+1}{2}-\binom{a+b+1}{2}}d_{-}^{a+b}z_1z_2\cdots z_{a}d_{+}^{a+b}(1)
\end{align*}
Since $\mathcal{N}y_i=z_i\mathcal{N}$ and $\mathcal{N}$ is an involution, it follows that $\mathcal{N}z_i=y_i\mathcal{N}$. Applying $\mathcal{N}$ to the evaluation above gives
\begin{align*}
\mathcal{N}(d_{-}^{a+b}y_1y_2\cdots y_a(d_{+}^{*})^{a+b}(1))&=\mathcal{N}((-1)^{a}t^{-a}q^{\binom{b+1}{2}-\binom{a+b+1}{2}}d_{-}^{a+b}z_1z_2\cdots z_{a}d_{+}^{a+b}(1))\\
&=(-1)^{a}t^{a}q^{\binom{a+b+1}{2}-\binom{b+1}{2}}d_{-}^{a+b}y_1y_2\cdots y_{a}(d_{+}^{*})^{a+b}(1)\\
&=(-1)^{a}t^{a}q^{\binom{a+b+1}{2}-\binom{b+1}{2}}(-1)^{a}q^{n(2^a1^b)}\omega\tilde{H}_{2^a1^b}[X;0,q^{-1}],
\end{align*}
where the final equality follows from Lemma \ref{lem-dplusdminustoHL}.

Comparing the two evaluations of $\mathcal{N}\big(d_{-}^{a+b}y_1y_2\cdots y_a(d_{+}^{*})^{a+b}(1)\big)$, we conclude that
\begin{align*}
(-1)^{a}\nabla\tilde{H}_{2^a1^b}[X;0,q]&=t^{a}q^{2\cdot n(2^a1^b)+\binom{a+b+1}{2}-\binom{b+1}{2}}\omega\tilde{H}_{2^a1^b}[X;0,q^{-1}]\\
&=t^{a}q^{\binom{2a+b}{2}+\binom{a+b}{2}}\omega\tilde{H}_{2^a1^b}[X;0,q^{-1}].
\end{align*}
Let $\Phi$ denote the map swapping $q$ and $t$. It is then evident that $\nabla\Phi=\Phi\nabla$. This completes the proof.
\end{proof}

\bibliographystyle{amsalpha}
\bibliography{Biblebib}

\end{document}